\documentclass[12pt,reqno]{amsart}
\usepackage{fullpage}
\usepackage{amssymb}
\usepackage{amsthm}
\usepackage{amsmath}
\usepackage[usenames]{color}
\usepackage{graphicx}
\usepackage[colorlinks=true,
linkcolor=webgreen,
filecolor=webbrown,
citecolor=webgreen]{hyperref}

\definecolor{webgreen}{rgb}{0,.5,0}
\definecolor{webbrown}{rgb}{.6,0,0}
\definecolor{red}{rgb}{1,0,0}

\usepackage{color}
\newtheorem{theorem}{Theorem}[section]

\newtheorem{lemma} [theorem]{Lemma}

\begin{document}

\title{On common values of $F_n$ and Nathanson's totient function $\Phi(m)$}

\author{Sagar Mandal$^{*}$}
\address{Department of Mathematics\\Indian Institute of Technology Ropar, Punjab, India.}
\email{sagar.25maz0008@iitrpr.ac.in, sagarmandal31415@gmail.com}
\thanks{$^{*}$The author is financially supported by the University Grants Commission (UGC), Government of India, through the award of the Junior Research Fellowship (JRF) with ref. No. 241620111598.}

\maketitle
\let\thefootnote\relax
\maketitle
\let\thefootnote\relax

    \begin{abstract}
    In a recent paper, Chatterjee, the author and Mohan posed the problem of determining all solutions of the Diophantine equation $F_n=\Phi(m)$, where $F_n$ is the $n$-th Fibonacci number and $\Phi(m)$ counts the number of nonempty sets $A \subseteq \{1, 2, \dots, m\}$ for which $\gcd(A)$ is relatively prime to $m$. In this paper, we prove that the Diophantine equation has the only  solutions $(n,m)=(1,1),(2,1),(3,2)$. The main tools used in this paper are lower bounds for linear forms in logarithms due to Matveev and Dujella-Peth{\H{o}} version of the Baker–Davenport reduction method in diophantine approximation.
\end{abstract}

\section{\protect\bigskip Introduction}
The Fibonacci sequence $\{F_n:n\geq0\}$ (\href{https://oeis.org/A000045}{\underline{A000045}}) is the binary recurrence sequence defined $F_0=0,F_1=1,$ and $ F_{n+2}=F_{n+1}+F_n$ for $n\geq0$. This is one of the most extensively studied sequences in number theory. In particular, Diophantine equations involving Fibonacci numbers and powers have attracted considerable attention over the years. A classical result of Bugeaud, Mignotte, and Siksek \cite{r4} completely determined the perfect powers appearing in the Fibonacci sequence. Later, Bravo and Luca \cite{r3} studied the equation $F_n+F_m=2^a$
and proved that it has only finitely many nonnegative integer solutions. Bravo and Bravo \cite{r1} extended this work to sums of three Fibonacci numbers, while Tiebekabe and Diouf \cite{r15,r16} considered sums of four and five Fibonacci numbers equal to powers of two. More generally, equations of the form $F_n\pm F_m=y^a$ have been investigated by several authors. In particular, Siar and Keskin \cite{r14} solved the equation $F_n-F_m=2^a,$
while Demirtürk Bitim and Keskin \cite{r6} studied the case $F_n-F_m=3^a.$
Further results for powers of five and for arbitrary perfect powers were obtained by Erduvan and Keskin \cite{r8}, Kebli, Kihel, Larone, Luca \cite{r10,r11}, and others. Recently, Jeong and Park \cite{r9} investigated the equation
$F_n^2+F_m^2=2^a$ and determined all its nonnegative integer solutions. Motivated by these developments and the increasing interest in exponential Diophantine equations involving Fibonacci numbers, the authors of \cite{Tmm} posed the problem of determining all solutions to the following equation:
\begin{align*}\label{eq1}
 F_n=\Phi(m)\tag{1}   
\end{align*}
where $\Phi(m)$ is the arithmetic function that counts the number of nonempty sets $A \subseteq \{1, 2, \dots, m\}$ for which $\gcd(A)$ is relatively prime to $m$.
Nathanson \cite{Nath} was the first to introduce and study the function $\Phi(m)$. Nathanson \cite{Nath} proved the following theorem using the M\H{o}bius inversion formula:

\begin{theorem}[\cite{Nath}, Theorem 3] \label{Nathanson Theorem}For all positive integers $m>1$, we have
    $$\Phi(m) = \sum_{d \mid m} \mu(d) 2^{m/d},$$
    where $\mu$ is the Möbius function. 
\end{theorem}
\noindent One can see that \href{https://oeis.org/A027375}{\underline{A027375}} enumerates $\Phi(m)$ for $m\geq2$. Generalizations of $\Phi(m)$ and several properties of the sequence $\{\Phi(m)\}_{m\geq1}$ have been investigated in \cite{T1,T2,T3,T4,Tmm,T5,T6}.
In this paper, we completely solve the equation (\ref{eq1}). Our main theorem is stated below. 
\begin{theorem}\label{Thm1}
The Diophantine equation (\ref{eq1}) has the only  solutions $(n,m)=(1,1),(2,1),(3,2)$.
\end{theorem}
\section{Preliminaries}
We start this section by recalling some basic notions from algebraic number theory.

Let $\eta $ be an algebraic number of degree $d$ with minimal polynomial 
$$a_{0}x^{d}+a_{1}x^{d-1}+\ldots+a_{d}=a_{0}\prod\limits_{i=1}^{d}\left(x-\eta
^{(i)}\right), $$
where the $a_{i}$'s are relatively prime integers with $a_{0}>0$ and $\eta
^{(i)}$'s are conjugates of $\eta $. Then 
\begin{align*}
  h(\eta )=\frac{1}{d}\left( \log a_{0}+\sum\limits_{i=1}^{d}\log \left( \max
\left\{ |\eta ^{(i)}|,1\right\} \right) \right)  
\end{align*}
is called the logarithmic height of $\eta .$ In particular, if $\eta =a/b$
is a rational number with $\gcd (a,b)=1$ and $b>0,$ then $h(\eta )=\log
\left( \max \left\{ |a|,b\right\} \right)$. The following are some well known properties of logarithmic height:
\begin{enumerate}
    \item[\upshape(1)] $h(\eta \pm \gamma )\leq h(\eta )+h(\gamma )+\log 2$,
    \item[\upshape(2)] $h(\eta \gamma ^{\pm 1})\leq h(\eta )+h(\gamma )$,
    \item[\upshape(3)]$h(\eta ^{s})=|s|h(\eta ),~s\in \mathbb{Z}$. 
\end{enumerate}
In order to prove Theorem \ref{Thm1}, we need the following theorem due to Matveev. 
\begin{theorem}[\cite{r13}]\label{Thm2}
Assume that $\gamma _{1},\gamma _{2},...,\gamma _{r}$ are
positive real algebraic numbers in a real algebraic number field $\mathbb{K}$
of degree $D$, $b_{1},b_{2},\ldots,b_{r}\in \mathbb{Z}$, and 
$$
\Lambda :=\gamma _{1}^{b_{1}}...\gamma _{r}^{b_{r}}-1 $$
is not zero. Then 
$$
|\Lambda |>\exp \left( -1.4\cdot 30^{r+3}\cdot r^{4.5}\cdot D^{2}(1+\log
D)(1+\log B)A_{1}A_{2}...A_{r}\right), $$
where 
$$
B\geq \max \left\{ |b_{1}|,...,|b_{r}|\right\}, 
$$
and 
$$A_{i}\geq \max \left\{ Dh(\gamma _{i}),|\log \gamma _{i}|,0.16\right\}\quad \text{for all } i=1,...,r. $$

\end{theorem}
We next present the following lemma which is an immediate variation of a result due to Dujella and Peth{\H{o}} \cite{r7}. In the following lemma, the function $||\cdot ||$ denotes the
distance from $x$ to the nearest integer, that is, $||x||=\min \left\{
|x-n|:n\in\mathbb{Z}\right\} $ for a real number $x.$

\begin{lemma}\label{Lem1}Let $M$ be a positive integer, let $p/q$ be a convergent of the
continued fraction expansion of the irrational number $\gamma $ such that $q>6M,$ and
let $A,B,\mu $ be some real numbers with $A>0$ and $B>1.$ Let $\epsilon
:=||\mu q||-M||\gamma q||.$ If $\epsilon >0,$ then there exists no solution
to the inequality 
$$
0<|u\gamma -v+\mu |<AB^{-w}, 
$$
in positive integers $u,v,$ and $w$ with 
$$
u\leq M\text{ and }w\geq \frac{\log (Aq/\epsilon )}{\log B}. 
$$
\end{lemma}

The following two properties of $\Phi(n)$ will play a pivotal role in the subsequent proofs.
\begin{theorem}[\cite{Tmm}]\label{thm for bound}
    Let $n\geq 4$ be an integer. Then $\Phi(n)\geq2^{n}-2^{n-2}$.
\end{theorem}
\begin{lemma}[\cite{T1,Tmm}]\label{Lemma 1} For every integer $n\geq3$, $6$ divides $ \Phi(n)$.  
\end{lemma}
It is well known that 
\begin{align*}\label{eq2}
F_{n}=\dfrac{\phi ^{n}-\psi ^{n}}{\sqrt{5}}   \quad \text{for  } n\geq0, \tag{2}
\end{align*}
where $\phi =\dfrac{1+\sqrt{5}}{2}$ and $\psi =\dfrac{1-\sqrt{5}}{2},$ which are the roots of the characteristic equation $x^{2}-x-1=0$. The following inequality can be proved by induction:
\begin{align*}\label{eq3}
    \phi ^{n-2}\leq F_{n}\leq \phi ^{n-1}  \quad \text{for  } n\geq1.\tag{3}
\end{align*}
\section{Proof of the main theorem}
To prove the main theorem, we require the following auxiliary lemma.
\begin{lemma}\label{Lemma 2} Let $n\geq 79$ be an integer. Then
$$\bigg| \sum_{\substack{d \mid n\\d>1}} \mu(d) 2^{n/d}\bigg|\leq 2^{2n/3}.$$    
\end{lemma}
\begin{proof} For $n\geq79$, we have $n-1\leq 2^{n/6}$ and $\omega(n)\leq \log_2n$, using these inequalities, we obtain
\begin{align*}
   \bigg| \sum_{\substack{d \mid n\\d>1}} \mu(d) 2^{n/d}\bigg|=\bigg|\sum_{\substack{d\mid \text{rad}(n)\\d>1}}\mu(d)2^{n/d}\bigg|\leq \sum_{\substack{d\mid \text{rad}(n)\\d>1}}2^{n/2}
   &\leq 2^{n/2}(2^{\omega(n)}-1)\\&\leq 2^{n/2}(n-1)\leq 2^{2n/3}.
\end{align*}
This completes the proof.
\end{proof}
\begin{proof}[Proof of Theorem \ref{Thm1}]
Let $(n,m)\in \mathbb{N}^2$ be a solution of the equation (\ref{eq1}).
Assume, for the sake of contradiction that $m\geq79$. By Lemma \ref{Lemma 1}, we observe that $6\mid \Phi(m)$, it follows that $6\mid F_n$. Since $6\mid F_n$ if and only if $12\mid n$, we write $n=12k$ for some $k\in \mathbb{N}$.
From Theorem \ref{thm for bound} and (\ref{eq3}), we have 
$$2^{m-1}<\Phi(m)=F_{12k}\leq \phi^{12k-1}<2^{12k-1},$$ 
which implies that $m<12k$. Now, using (\ref{eq2}), the equation (\ref{eq1}) can be written as
$$\frac{\phi^{12k}-\psi^{12k}}{\sqrt{5}}=2^m+\sum_{\substack{d \mid m\\d>1}} \mu(d) 2^{m/d}.$$
For $m\geq79$, from the above equation and by Lemma \ref{Lemma 2}, we get
\begin{align*}
\bigg|\phi^{12k}-\sqrt{5}\cdot2^m\bigg|&=\bigg|\psi^{12k}+\sqrt{5}\sum_{\substack{d \mid m\\d>1}} \mu(d) 2^{m/d}\bigg|\leq 1+\sqrt{5}\cdot 2^{2m/3}\leq \frac{9}{4}\cdot 2^{2m/3}.
\end{align*}
After multiplying both sides of the above inequality by $\dfrac{1}{\phi^{12k}}$, we obtain
\begin{align*}\label{eq4}
    \bigg|1-\frac{\sqrt{5}\cdot 2^m}{\phi^{12k}}\bigg|\leq \frac{9\cdot 2^{2m/3}}{4\cdot\phi^{12k}}.\tag{4}
\end{align*}

Now, let us apply Theorem \ref{Thm2} with $(\gamma _{1},b_1)=(\sqrt{5},1),~(\gamma _{2},b_2)=(2 ,m),~(\gamma_{3},b_3)=(\phi,-12k)$, and 
$$\Lambda = \frac{\sqrt{5}\cdot 2^m}{\phi^{12k}} - 1.$$
Note that $\gamma_1,\gamma_2,\gamma_3\in \mathbb{Q}(\sqrt{5})$, therefore, we take $D=[\mathbb{Q}(\sqrt5):\mathbb{Q}]=2$.
Observe that $\Lambda\neq0$, as if $\Lambda=0$ then we have $\phi^{12k}=\sqrt{5}\cdot 2^m$, now conjugating the equation in $\mathbb{Q}(\sqrt{5})$, we get $\psi^{12k}=-\sqrt{5}\cdot 2^m=-\phi^{12k}$, which implies that $\phi^{12k}+\psi^{12k}=L_{12k}=0$, where $L_{12k}$ is the $12k$-th Lucas number, but this is impossible. Set
$$B=12k\geq\max\{|b_1|,|b_2|,|b_3|\},$$
$$ A_1=1.7\geq \max\{D h(\gamma_1),\,|\log \gamma_1|,\,0.16\}=2\log\sqrt{5},$$
$$A_2=1.4\geq \max\{D h(\gamma_2),\,|\log \gamma_2|,\,0.16\}=2\log2,$$
$$A_3=0.49\geq \max\{D h(\gamma_3),\,|\log \gamma_3|,\,0.16\}=\log \phi.$$

Then, by Theorem \ref{Thm2} and (\ref{eq4}), we obtain 
\begin{align*}
    \exp(-1.4\cdot 30^{6}\cdot 3^{4.5}\cdot D^2(1+\log D)\cdot A_1\cdot A_2\cdot A_3\cdot (1+\log 12k))<|\Lambda|\leq\frac{9\cdot 2^{2m/3}}{4\phi^{12k}}.
\end{align*}
It follows that
\begin{align*}
-1.4\cdot 30^{6}\cdot 3^{4.5}\cdot D^2(1+\log D)\cdot A_1\cdot A_2\cdot A_3\cdot(1+\log 12k)\leq \log(9/4)+\frac{2m}{3}\log2-12k\log \phi,
\end{align*}
from which we write
\begin{align*}
12k\log \phi&\leq \log(9/4)+\frac{2m}{3}\log2+1.4\cdot 30^{6}\cdot 3^{4.5}\cdot D^2(1+\log D)\cdot A_1\cdot A_2\cdot A_3\cdot(1+\log 12k)\\
&\leq\log(9/4)+8k\log2+1.4\cdot 30^{6}\cdot 3^{4.5}\cdot D^2(1+\log D)\cdot A_1\cdot A_2\cdot A_3\cdot(1+\log12k)\\
&\leq \log(9/4)+8k\log2+12\cdot10^{11}(1+\log 12k).
\end{align*}
The above inequality implies that $k<1.91\cdot 10^{14}$. We now need to reduce the upper bound obtained for $k$. For this, we use Lemma \ref{Lem1}.\\
Let $z=\log\sqrt{5}+m\log 2-12k\log \phi$. Then $z\neq 0$. If $z>0$, then from (\ref{eq4}), we have
$$0<z\leq e^{z}-1<\frac{9\cdot 2^{2m/3}}{4\phi^{12k}}.$$
If $z<0$, then from (\ref{eq4}), we have
$$|e^{z}-1|<\frac{9\cdot 2^{2m/3}}{4\phi^{12k}}<\frac{1}{2}, \quad \text{for }k> m/12.$$
Recall that for any nonzero real number $x<0$, if $|e^{x}-1| <\dfrac{1}{2}$, then $
|x| < 2|e^{x}-1|$. Therefore, for any $k> m/12$, we have 
$$|z|<2|e^{z}-1|<\frac{4.5\cdot 2^{2m/3}}{\phi^{12k}}.$$
From the above inequality, we get
\begin{align*}
\bigg|\frac{\log\sqrt{5}}{\log\phi}+\frac{m\log2}{\log\phi}-12k\bigg|<\frac{4.5\cdot 2^{2m/3}\cdot\phi^{-n}}{\log\phi}<\frac{4.5\cdot 2^{8k}\cdot\phi^{-{12k}}}{\log\phi}&=\frac{4.5\cdot (\dfrac{4}{1.5^3})^{4k}}{\log\phi}\cdot (\frac{\phi}{1.5})^{-12k}\label{eq5}\tag{5}\\
&\leq 6.77\cdot 10^{56372667618942}\cdot (\frac{\phi}{1.5})^{-12k}, 
\end{align*}
the last inequality follows from the fact that $k<1.91\cdot 10^{14}$.
Now, let us apply Lemma \ref{Lem1} with
\begin{align*}
    u=m,\quad \gamma=\frac{\log2}{\log \phi},\quad v=12k,\quad \mu=\frac{\log\sqrt{5}}{\log \phi},\quad A=6.77\cdot 10^{56372667618942},\quad B=\frac{\phi}{1.5},\quad w=12k.
\end{align*}
Since $k< 1.91\cdot  10^{14}$ and $u=m\leq 12k$, we set $M=22.92\cdot 10^{14}$. Choosing  
$$\frac{p}{q}=\frac{p_{37}}{q_{37}}=\frac{78462338394551841}{54471843954966727},$$
we get $q_{37}>6M$ and
$$\epsilon=||\mu q||-M||\gamma q||=0.48725\ldots>0.$$
Therefore, by Lemma \ref{Lem1}, we have
\begin{align*}
    12k<\frac{\log(Aq/\epsilon)}{\log B}<1713643433482951.
\end{align*}
It implies that $k\leq 142803619456912$.

We now substitute the new upper bound for $k$ into the inequality (\ref{eq5}), to get a new better upper bound for $k$. We iterate this procedure until no further improvement of the upper bound for $k$ is possible.

Now, using $k\leq 142803619456912$, from (\ref{eq5}), we obtain
\begin{align*}
\bigg|\frac{\log\sqrt{5}}{\log\phi}+\frac{m\log2}{\log\phi}-12k\bigg|<\frac{4.5\cdot 2^{2m/3}\cdot\phi^{-n}}{\log\phi}<\frac{4.5\cdot 2^{8k}\cdot\phi^{-{12k}}}{\log\phi}&=\frac{4.5\cdot (\dfrac{4}{1.5^3})^{4k}}{\log\phi}\cdot (\frac{\phi}{1.5})^{-12k}\\
&\leq 1.36\cdot 10^{42147753792809}\cdot (\frac{\phi}{1.5})^{-12k} 
\end{align*}
In Lemma \ref{Lem1}, we set
\begin{align*}
    u=m,\quad \gamma=\frac{\log2}{\log \phi},\quad v=12k,\quad \mu=\frac{\log\sqrt{5}}{\log \phi},\quad A=1.36\cdot 10^{42147753792809},\quad B=\frac{\phi}{1.5},\quad w=12k.
\end{align*}
Since $k\leq 142803619456912$ and $u=m\leq 12k$, we take $M=1713643433482944$. Choosing  
$$\frac{p}{q}=\frac{p_{37}}{q_{37}}=\frac{78462338394551841}{54471843954966727},$$
we get $q_{37}>6M$ and
$$\epsilon=||\mu q||-M||\gamma q||=0.48725\ldots>0.$$
Therefore, by Lemma \ref{Lem1}, we have
\begin{align*}
    12k<\frac{\log(Aq/\epsilon)}{\log B}<1281227668900324.
\end{align*}
From the above inequality, we get $k\leq 106768972408360$. 

Using this new upper bound obtained for $k$ to (\ref{eq5}), we have
\begin{align*}
\bigg|\frac{\log\sqrt{5}}{\log\phi}+\frac{m\log2}{\log\phi}-12k\bigg|<\frac{4.5\cdot 2^{2m/3}\cdot\phi^{-n}}{\log\phi}<\frac{4.5\cdot 2^{8k}\cdot\phi^{-{12k}}}{\log\phi}&=\frac{4.5\cdot (\dfrac{4}{1.5^3})^{4k}}{\log\phi}\cdot (\frac{\phi}{1.5})^{-12k}\\
&\leq 6.89\cdot 10^{31512313055458}\cdot (\frac{\phi}{1.5})^{-12k} 
\end{align*}
Now, let us apply Lemma \ref{Lem1} with
\begin{align*}
    u=m,\quad \gamma=\frac{\log2}{\log \phi},\quad v=12k,\quad \mu=\frac{\log\sqrt{5}}{\log \phi},\quad A=6.89\cdot 10^{31512313055458},\quad B=\frac{\phi}{1.5},\quad w=12k.
\end{align*}
Since $k\leq  106768972408360$ and $u=m\leq 12k$, we set $M=1281227668900320$. Choosing  
$$\frac{p}{q}=\frac{p_{37}}{q_{37}}=\frac{78462338394551841}{54471843954966727},$$
we get $q_{37}>6M$ and
$$\epsilon=||\mu q||-M||\gamma q||=0.48725\ldots>0.$$
Therefore, by Lemma \ref{Lem1}, we have
\begin{align*}
    12k<\frac{\log(Aq/\epsilon)}{\log B}<957926431766255,
\end{align*}
which implies that $k\leq 79827202647187$. Again using $k\leq 79827202647187$, we get from (\ref{eq5}) that 
\begin{align*}
\bigg|\frac{\log\sqrt{5}}{\log\phi}+\frac{m\log2}{\log\phi}-12k\bigg|<\frac{4.5\cdot 2^{2m/3}\cdot\phi^{-n}}{\log\phi}<\frac{4.5\cdot 2^{8k}\cdot\phi^{-{12k}}}{\log\phi}&=\frac{4.5\cdot (\dfrac{4}{1.5^3})^{4k}}{\log\phi}\cdot (\frac{\phi}{1.5})^{-12k}\\
&\leq 3.5\cdot 10^{23560588281570}\cdot (\frac{\phi}{1.5})^{-12k} 
\end{align*}
In Lemma \ref{Lem1} we take
\begin{align*}
    u=m,\quad \gamma=\frac{\log2}{\log \phi},\quad v=12k,\quad \mu=\frac{\log\sqrt{5}}{\log \phi},\quad A=3.5\cdot 10^{23560588281570},\quad B=\frac{\phi}{1.5},\quad w=12k.
\end{align*}
Since $k\leq  79827202647187$ and $u=m\leq 12k$, we take $M=957926431766244$. Choosing  
$$\frac{p}{q}=\frac{p_{36}}{q_{36}}=\frac{9072992347323886}{6298851569562081},$$
we get $q_{36}>6M$ and
$$\epsilon=||\mu q||-M||\gamma q||=0.2060\ldots>0.$$
Therefore, by Lemma \ref{Lem1}, we have
\begin{align*}
    12k<\frac{\log(Aq/\epsilon)}{\log B}<716206081831009,
\end{align*}
from this we obtain $k\leq 59683840152584$.

Now, iterating this procedure using a Python code \cite{Sagar} for $108$ steps, we obtain $k\leq 163$.

After that, no further improvement of the upper bound is possible.
After performing a brute-force search using a Python code \cite{Sagar} for integers $k$ satisfying $6\leq m/12<k\leq 163$, we found that equation (\ref{eq1}) has no solutions for $79\leq m<12k$.

Hence, we must have $1\leq m\leq 78$.
For $3\leq m\leq 78$, again performing a brute-force search using a Python code \cite{Sagar} for all integers $m$ satisfying $3\leq m\leq 78$, we found that equation (\ref{eq1}) has no solutions for $k$.

The proof now follows by observing that $\Phi(2)=2=F_3$ and $\Phi(1)=1=F_1=F_2$.
\end{proof}

\section{Data Availability} 	
The author confirms that the manuscript has no associated data.

\section{Competing Interests}
The author confirms that he has no competing interest.

\end{document}